\documentclass[envcountsame]{llncs}
\usepackage{amssymb}
\usepackage{amsmath}
\usepackage{amsfonts}
\usepackage{makeidx}
\usepackage{enumerate}

\RequirePackage{amsfonts,amssymb}

\begin{document}

\title{Cohesive Powers of Linear Orders}
\author{Rumen Dimitrov\inst{1} \and Valentina Harizanov\inst{2} \and Andrey
Morozov\inst{3} \and Paul Shafer\inst{4} \and Alexandra Soskova\inst{5} \and Stefan Vatev\inst{6}}
\institute{Department of Mathematics, Western Illinois University, Macomb, IL 61455, USA \email{rd-dimitrov@wiu.edu} \and
Department of Mathematics, George Washington University,  Washington, DC 20052, USA \email{harizanv@gwu.edu} \and
Sobolev Institute of Mathematics,  Novosibirsk, 630090, Russia \email{morozov@math.nsc.ru} \and
School of Mathematics, University of Leeds, Leeds LS2 9JT, United Kingdom \email{P.E.Shafer@leeds.ac.uk} \and
Faculty of Mathematics and Informatics, Sofia University, 5 James Bourchier blvd., 1164, Sofia, Bulgaria  \email{asoskova@fmi.uni-sofia.bg} \and
Faculty of Mathematics and Informatics, Sofia University, 5 James Bourchier blvd., 1164, Sofia, Bulgaria  \email{stefanv@fmi.uni-sofia.bg}
\thanks{The first three and the last two authors acknowledge partial support of the NSF grant DMS-1600625. The second author acknowledges support from the Simons Foundation Collaboration Grant, and from CCFF and Dean's Research Chair GWU awards. The last two authors acknowledge support from BNSF, MON, DN 02/16. The fourth author acknowledges the support of the \emph{Fonds voor Wetenschappelijk Onderzoek -- Vlaanderen} Pegasus program.}}
\thanks{MSC-class: 03C57 (Primary) 03D45, 03C20 (Secondary)}

\maketitle

\begin{abstract}
Cohesive powers of computable structures can be viewed as effective
ultraproducts over effectively indecomposable sets called cohesive
sets. We investigate the isomorphism types of cohesive powers $\Pi _{C}%
\mathcal{L}$ for familiar computable linear orders $\mathcal{L}$. If $%
\mathcal{L}$ is isomorphic to the ordered set of natural numbers $\mathbb{N}$
and has a computable successor function, then $\Pi _{C}\mathcal{L}$ is
isomorphic to $\mathbb{N}+\mathbb{Q}\times \mathbb{Z}.$ Here, $+$ stands for
the sum and $\times $ for the lexicographical product of two orders. We
construct computable linear orders $\mathcal{L}_{1}$ and $\mathcal{L}_{2}$
isomorphic to $\mathbb{N},$ both with noncomputable successor functions,
such that $\Pi _{C}\mathcal{L}_{1}\mathbb{\ }$is isomorphic to $\mathbb{N}+%
\mathbb{Q}\times \mathbb{Z}$, while $\Pi _{C}\mathcal{L}_{2}$ is not$.$
While cohesive powers preserve all $\Pi _{2}^{0}$ and $\Sigma _{2}^{0}$ sentences, we provide new
examples of $\Pi _{3}^{0}$ sentences $\Phi $ and computable structures $%
\mathcal{M}$ such that $\mathcal{M}\vDash \Phi $ while $\Pi _{C}\mathcal{M}%
\vDash \urcorner \Phi .$
\end{abstract}

\section{Introduction and Preliminaries}

Skolem was the first to construct a countable nonstandard model of true
arithmetic. Various countable nonstandard models of (fragments of)
arithmetic have been later studied by Feferman, Scott, Tennenbaum,
Hirschfeld, Wheeler, Lerman, McLaughlin and others (see \cite{FST}, \cite{L}%
, \cite{HW}, \cite{M}). The following definition, and other notions from
computability theory can be found in \cite{RS}.

\begin{definition}
(i) An infinite set $C\subseteq \omega $ is \emph{cohesive} (\emph{r-cohesive%
}) if for every c.e.\ (computable) set $W,$ either $W\cap C$ or $\overline{W}%
\cap C$ is finite.

(ii) A set $M$ is \emph{maximal} (\emph{r-maximal}) if $M$ is c.e.\ and $%
\overline{M}$ is cohesive (\emph{r}-cohesive).

(iii) If  $M$ is maximal, then $\overline{M}$ is called co-maximal.

(iv) A set $B$ is \emph{quasimaximal} if it is the intersection of finitely
many maximal sets.
\end{definition}

In the definition above $\omega $ denotes the set of natural numbers. We
will use $=^{\ast }$(and $\subseteq ^{\ast })$ to refer to equality
(inclusion) of sets up to finitely many elements. Let $A$ be a fixed \emph{r}%
-cohesive set. For computable functions $f$ and $g,$\ Feferman, Scott, and
Tennenbaum (see \cite{FST}) defined an equivalence relation $f\sim _{A}g$ if
$A\subseteq ^{\ast }\{n:f(n)=g(n)\}.$ They then proved that the structure $%
\mathcal{R}/\!\sim_{A},$ with domain the set of recursive functions modulo $%
\sim _{A},$ is a model of only a fragment of arithmetic. They constructed a
particular $\Pi _{3}^{0}$ sentence $\Phi $ such that for the standard model of
arithmetic, $\mathcal{N},$ we have $\mathcal{N}\vDash \Phi $ but $\mathcal{R}/\!\sim_{A}\nvDash \Phi . $ The
sentence $\Phi $ provided in \cite{FST} essentially uses Kleene's $T$
predicate.

\emph{Cohesive powers }of computable structures are effective versions of
ultrapowers. They have been introduced in \cite{D2} in relation to the study
of automorphisms of the lattice $\mathcal{L}^{\ast }(V_{\infty })$ of
effective vector spaces$.$ Cohesive powers of the field of rational numbers
were used in \cite{D1} to characterize certain principal filters of $%
\mathcal{L}^{\ast }(V_{\infty })$. Their isomorphism types and automorphisms
were further studied in \cite{DHMM}$.$ They were also used in \cite{D1} and %
\cite{DH} to find interesting orbits in $\mathcal{L}^{\ast }(V_{\infty })$.

The goal of this paper is to show that the presentation of a computable
structure matters for the isomorphism type of its cohesive power. We give
computable presentations of the ordered set of natural numbers such that
their cohesive powers are not elementary equivalent. Furthermore, we provide
examples of computable structures $\mathcal{M}$ and $\Pi _{3}^{0}$ sentences
$\Psi ,$ which do not use Kleene's $T$ predicate, such that $\mathcal{M}%
\vDash \Psi $ while the cohesive power $\Pi _{C}\mathcal{M}\vDash \urcorner
\Psi .$ We will now present some additional definitions and known results.

\begin{definition}
\cite{D2} Let $\mathcal{A}$ be a computable structure for a computable
language $L$ and with domain $A$. Let $C\subseteq \omega $ be a cohesive
set. The \emph{cohesive power of }$\mathcal{A}$\emph{\ over }$C$, denoted by
$\Pi _{C}\mathcal{A}$, is a structure $\mathcal{B}$ for $L$ defined as
follows:

(i) Let $D=\{\varphi \mid \varphi :\omega \rightarrow A$ is a partial
computable function, and $C\subseteq ^{\ast }dom(\varphi )\}$.

For $\varphi _{1},\varphi _{2}\in D$, define $\varphi _{1}=_{C}\varphi
_{_{2}}$ iff $C\subseteq ^{\ast }\{x:\varphi _{1}(x){\downarrow} =\varphi
_{2}(x){\downarrow} \}$.

Let $B=(D/=_{C})$ be the domain of $\mathcal{B=}\Pi _{C}\mathcal{A}$

(ii) If $f\in L$ is an $n$-ary function symbol, then $f^{\mathcal{B}}$ is an
$n$-ary function on $B$ such that for every $[\varphi _{1}],\ldots ,[\varphi
_{n}]\in B,$ $f^{\mathcal{B}}([\varphi _{1}],\ldots ,[\varphi
_{n}])=[\varphi ]$, where for every $x\in A$,
\begin{equation*}
\varphi (x)\simeq f^{\mathcal{A}}(\varphi _{1}(x),\ldots ,\varphi _{n}(x)),
\end{equation*}%
where $\simeq $ stands for equality of partial functions.

(iii) If $P\in L$ is an $m$-ary predicate symbol, then $P^{\mathcal{B}}$ is
an $m$-ary relation on $B$ such that for every $[\varphi _{1}],\ldots
,[\varphi _{m}]\in B$,
\begin{equation*}
P^{\mathcal{B}}([\varphi _{1}],\ldots ,[\varphi _{m}]) \Leftrightarrow
C\subseteq ^{\ast }\{x\in A\mid P^{\mathcal{A}}(\varphi _{1}(x),\ldots
,\varphi _{m}(x))\}.
\end{equation*}%
$\smallskip $(iv) If $c\in L$ is a constant symbol, then $c^{\mathcal{B}}$
is the equivalence class of the (total) computable function on $A$ with
constant value $c^{\mathcal{A}}$.
\end{definition}

The following is the fundamental theorem of cohesive powers due to Dimitrov (see \cite{D2}).

\begin{theorem}
\label{FTCP} Let $C$ be  a cohesive set and let $\mathcal{A}$ and $\mathcal{B}$ be as in the definition
above.

\begin{enumerate}
\item If $\tau (y_{1},\ldots ,y_{n})$ is a term in $L$ and $[\varphi
_{1}],\ldots ,[\varphi _{n}]\in B,$ then $[\tau ^{\mathcal{B}}([\varphi
_{1}],\ldots ,[\varphi _{n}])]$ is the equivalence class of a partial
computable function such that
\begin{equation*}
\tau ^{\mathcal{B}}([\varphi _{1}],\ldots ,[\varphi _{n}])(x)=\tau ^{%
\mathcal{A}}(\varphi _{1}(x),\ldots ,\varphi _{n}(x)).
\end{equation*}

\item If $\Phi (y_{1},\ldots ,y_{n})$ is a formula in $L$ that is a Boolean
combination of $\Sigma _{1}^{0}$ and $\Pi _{1}^{0}$ formulas and $[\varphi
_{1}],\ldots ,[\varphi _{n}]\in B,$ then%
\begin{equation*}
\mathcal{B}\vDash \Phi ([\varphi _{1}],\ldots ,[\varphi _{n}])\text{ iff }%
C\subseteq ^{\ast }\{x:\mathcal{A}\models \Phi (\varphi _{1}(x),\ldots
,\varphi _{n}(x))\}.
\end{equation*}

\item If $\Phi $ is a $\Pi _{2}^{0}$ (or $\Sigma _{2}^{0})$ sentence in $L$,
then $\mathcal{B}\vDash \Phi $ iff $\mathcal{A}\vDash \Phi .\smallskip $

\item If $\Phi $ is a $\Pi _{3}^{0}$ sentence in $L$, then $\mathcal{B}%
\vDash \Phi $ implies $\mathcal{A}\vDash \Phi .$\medskip
\end{enumerate}
\end{theorem}

Note that $\mathcal{A}$ is a substructure of $\mathcal{B=}\Pi _{C}\mathcal{A}
$. For $c\in A$ let $[\varphi _{c}]\in B$ be the equivalence class of the
total function $\varphi _{c}$ such that $\varphi _{c}(x)=c$ for every $x\in
\omega $. The map $d:A\rightarrow B$ such that $d(c)=[\varphi _{c}]$ is
called \emph{canonical embedding} of $\mathcal{A}$ into $\mathcal{B}$.

\section{Cohesive Powers of Linear Orders}

We will now investigate various algebraic and computability-theoretic
properties of cohesive powers of linear orders. We first provide some
definitions and notational conventions we will use. Let $C\subseteq \omega $
be a cohesive set. Let $\left\langle \cdot ,\cdot \right\rangle :\omega
^{2}\rightarrow \omega $ be a fixed computable bijection, and let the
(computable) functions $\pi _{1}$ and $\pi _{2}$ be such that $\pi
_{1}(\left\langle m,n\right\rangle )=m$ and $\pi _{2}(\left\langle
m,n\right\rangle )=n.$

\begin{definition}
\label{product_LO}Let $\mathcal{L}_{0}=\left\langle L_{0},<_{\mathcal{L}%
_{0}}\right\rangle $ and $\mathcal{L}_{1}=\left\langle L_{1},<_{\mathcal{L}%
_{1}}\right\rangle $ be linear orders. Then
\end{definition}

(1) $\mathcal{L}_{0}+\mathcal{L}_{1}=\left\langle \{\left\langle
0,l\right\rangle :l\in L_{0}\}\cup \{\left\langle 1,l\right\rangle :l\in
L_{1}\},<_{\mathcal{L}_{0}+\mathcal{L}_{1}}\right\rangle ,$ where

\begin{equation*}
\left\langle i,l\right\rangle <_{\mathcal{L}_{0}+\mathcal{L}%
_{1}}\left\langle j,m\right\rangle \ \text{iff\ }\left( i<j\right) \vee
\left( i=j\wedge l<_{\mathcal{L}_{i}}m\right) .
\end{equation*}

(2) $\mathcal{L}_{0}\times \mathcal{L}_{1}=\left\langle L_{0}\times L_{1},<_{%
\mathcal{L}_{0}\times \mathcal{L}_{1}}\right\rangle ,$ where

$\left\langle k,m\right\rangle <_{\mathcal{L}_{0}\times \mathcal{L}%
_{1}}\left\langle l,n\right\rangle $ iff $\left( k<_{\mathcal{L}%
_{0}}l\right) \vee \left( k=_{\mathcal{L}_{0}}l\wedge m<_{\mathcal{L}%
_{1}}n\right) .$

\begin{remark}
(1) By $\mathbb{N}$, $\mathbb{Z}$, and $\mathbb{Q}$ we denote the usual
ordered sets of natural numbers, integers, and rational numbers. The order
types of $\mathbb{N}$, $\mathbb{Z}$, and $\mathbb{Q}$ are denoted by $\omega
,$ $\zeta $, and $\eta .$

(2) In Definition \ref{product_LO} we use $\mathcal{L}_{0}\times \mathcal{L}%
_{1}$ to denote the lexicographical product of the linear orders $\mathcal{L}%
_{0}$ and $\mathcal{L}_{1}.$ This product is also denoted by $\mathcal{L}%
_{1}\cdot \mathcal{L}_{0}.$ (For example, $\mathbb{Q}\times \mathbb{Z}$ is
also denoted by $\mathbb{Z\cdot Q}$, and its order type is denoted $\zeta
\cdot \eta .$)

(3) We will use $\mathcal{L}^{rev}$ to denote the reverse linear order of $%
\mathcal{L}$. (In the literature it is also denoted by $\mathcal{L}^{\ast }$%
.)

(4) Let the quantifier $\forall ^{\infty }n$ stand for ``infinitely many $n.$'' Note that if $\{n \mid \varphi (n)\}$ is a c.e.~set, then $\left( \forall ^{\infty }n\in C\right) \left[ \varphi (n)\right] $ will mean that $\varphi (n)$ is satisfied ''for almost all $n\in C$.''

\end{remark}

Before we state the next theorem, we would like to remind that $\mathbb{N} +
\mathbb{Q} \times \mathbb{Z}$ is the order type of a countable non-standard
model of PA.

\begin{theorem}
\label{Properties_LO}Let $\mathcal{L}_{0}$ and $\mathcal{L}_{1}$ be
computable linear orders and let $C$ be a cohesive set. Then

(1) $\Pi _{C}\left( \mathcal{L}_{0}+\mathcal{L}_{1}\right) \cong \Pi _{C}%
\mathcal{L}_{0}+\Pi _{C}\mathcal{L}_{1}$

(2) $\Pi _{C}\left( \mathcal{L}_{0}\times \mathcal{L}_{1}\right) \cong \Pi
_{C}\mathcal{L}_{0}\times \Pi _{C}\mathcal{L}_{1}$

(3) $\Pi _{C}\mathcal{L}_{0}^{rev}\cong \left( \Pi _{C}\mathcal{L}%
_{0}\right) ^{rev}$

(4) Let $\mathcal{A}$ be a computable presentation of the linear order $%
\mathbb{N}$ with a computable successor function. Then $\Pi _{C}\mathcal{A}%
\cong \mathbb{N+Q}\times \mathbb{Z}.$

(5) If $\mathcal{L}$ is a computable dense linear order without endpoints,
then $\mathcal{L}\cong \Pi _{C}\mathcal{L}.$
\end{theorem}

\begin{proof}
(1) Let $\mathcal{A}=\Pi _{C}\left( \mathcal{L}_{0}+\mathcal{L}_{1}\right) $
and $\mathcal{B}=\Pi _{C}\mathcal{L}_{0}+\Pi _{C}\mathcal{L}_{1}.$ We will
define an isomorphism $\Phi :\mathcal{A}\rightarrow \mathcal{B}.$ Suppose $%
\left[ \varphi \right] _{C}\in \Pi _{C}\left( \mathcal{L}_{0}+\mathcal{L}%
_{1}\right) $ for a partial computable function $\varphi $.$\smallskip $

If $\left( \forall ^{\infty }n\in C\right) \left[ \varphi (n)\in \left\{
0\right\} \times L_{1}\right] $, then let $\Phi (\left[ \varphi \right]
_{C})=_{def}\left\langle 0,\left[ \pi _{2}\circ \varphi \right]
_{C}\right\rangle .\smallskip $

If $\left( \forall ^{\infty }n\in C\right) \left[ \varphi (n)\in \left\{
1\right\} \times L_{2}\right] ,$ then let $\Phi (\left[ \varphi \right]
_{C})=_{def}\left\langle 1,\left[ \pi _{2}\circ \varphi \right]
_{C}\right\rangle .\smallskip $

Since $C$ is cohesive, exactly one of the two cases above applies, so it
follows that so it follows that $\Phi$ is well defined. It is then easy to
check that $\Phi$ is an isomorphism.\medskip

(2) Let $\mathcal{A}=\Pi _{C}\left( \mathcal{L}_{0}\times \mathcal{L}%
_{1}\right) $ and $\mathcal{B}=\Pi _{C}\mathcal{L}_{0}\times \Pi _{C}%
\mathcal{L}_{1}.$ We will define an isomorphism $\Phi :\mathcal{A}%
\rightarrow \mathcal{B}.$ Suppose $\left[ \varphi \right] _{C}\in \Pi
_{C}\left( \mathcal{L}_{0}\times \mathcal{L}_{1}\right) ,$ and let $\Phi (%
\left[ \varphi \right] _{C})=_{def}\left\langle \left[ \pi _{1}\circ \varphi %
\right] _{C},\left[ \pi _{2}\circ \varphi \right] _{C}\right\rangle .$ We
will prove that\
\begin{equation*}
\left[ \varphi \right] _{C}<_{\mathcal{A}}\left[ \psi \right]
_{C}\Leftrightarrow \left\langle \left[ \pi _{1}\circ \varphi \right] _{C},%
\left[ \pi _{2}\circ \varphi \right] _{C}\right\rangle <_{\mathcal{B}%
}\left\langle \left[ \pi _{1}\circ \psi \right] _{C},\left[ \pi _{2}\circ
\psi \right] _{C}\right\rangle .
\end{equation*}%
By definition, $\left[ \varphi \right] _{C}<_{\mathcal{A}}\left[ \psi \right]
_{C}$ iff $C\subseteq ^{\ast }\{n:\varphi (n)<\psi (n)\}.$ By cohesiveness
of $C,$ we will have either\smallskip

$\left( \forall ^{\infty }n\in C\right) [\left( \pi _{1}\circ \varphi
\right) (n)<\left( \pi _{1}\circ \psi \right) (n)],$ or

$\left( \forall ^{\infty }n\in C\right) \left[ \left( \pi _{1}\circ \varphi
\right) (n)=\left( \pi _{1}\circ \psi \right) (n)\wedge \left( \pi _{2}\circ
\varphi \right) (n)<\left( \pi _{2}\circ \psi \right) (n)\right] .$

\noindent In the first case, $\left[ \pi _{1}\circ \varphi \right]
_{C}<_{\Pi _{C}\mathcal{L}_{0}}\left[ \pi _{1}\circ \psi \right] _{C}.$ In
the second case, $\left[ \pi _{1}\circ \varphi \right] _{C}=_{\Pi _{C}%
\mathcal{L}_{0}}\left[ \pi _{1}\circ \psi \right] _{C}$ and $\left[ \pi
_{2}\circ \varphi \right] _{C}<_{\Pi _{C}\mathcal{L}_{1}}\left[ \pi
_{2}\circ \psi \right] _{C}.$ Therefore,
\begin{equation*}
\left\langle \left[ \pi _{1}\circ \varphi \right] _{C},\left[ \pi _{2}\circ
\varphi \right] _{C}\right\rangle <_{\mathcal{B}}\left\langle \left[ \pi
_{1}\circ \psi \right] _{C},\left[ \pi _{2}\circ \psi \right]
_{C}\right\rangle .
\end{equation*}

(3) Let $\mathcal{A}=\Pi _{C}\mathcal{L}_{0}^{rev}$ and $\mathcal{B}=\left(
\Pi _{C}\mathcal{L}_{0}\right) ^{rev}.$ We will define an isomorphism $\Phi :%
\mathcal{A}\rightarrow \mathcal{B}.$ If $\left[ \varphi \right] _{C}\in \Pi
_{C}\mathcal{L}_{0}^{rev},$ then let $\Phi \left( \left[ \varphi \right]
_{C}\right) =\left[ \varphi \right] _{C}.$ We will prove that $\left[
\varphi \right] _{C}<_{\mathcal{A}}\left[ \psi \right] _{C}$ iff $\left[
\varphi \right] _{C}<_{\mathcal{B}}\left[ \psi \right] _{C}.$ By definition
we have
\begin{equation*}
\left[ \varphi \right] _{C}<_{\mathcal{B}}\left[ \psi \right]
_{C}\Leftrightarrow \left[ \psi \right] _{C}<_{\Pi _{C}\mathcal{L}_{0}}\left[
\varphi \right] _{C}\Leftrightarrow
\end{equation*}%
\begin{equation*}
\left( \forall ^{\infty }n\in C\right) (\psi (n)<_{\mathcal{L}_{0}}\varphi
(n))\Leftrightarrow
\end{equation*}%
\begin{equation*}
\left( \forall ^{\infty }n\in C\right) (\varphi (n)<_{\mathcal{L}%
_{0}^{rev}}\psi (n))\Leftrightarrow \left[ \varphi \right] _{C}<_{\mathcal{A}%
}\left[ \psi \right] _{C}.
\end{equation*}

(4) The proof of this fact is omitted because it is a simplified version of
the proof of Theorem \ref{noncomputable_succ}.

(5) The theory of dense linear orders without endpoints is $\Pi _{2}^{0}$
axiomatizable and countably categorical. By Theorem \ref{FTCP} (part 4), $%
\Pi _{C}\mathcal{L}$ is also a dense linear order without endpoints. Since $%
\Pi _{C}\mathcal{L}$ is countable, we have $\mathbb{Q}\cong \mathcal{L}\cong
\Pi _{C}\mathcal{L}.$
\end{proof}

Item (5) in the previous Theorem provides an example of an infinite
structure $\mathcal{L}$ such that $\mathcal{L}\cong \Pi _{C}\mathcal{L}$.
The linear order $\mathbb{Q}$ is an ultrahomogeneous structure; it is the Fra%
\"{\i}ss\'{e} limit of the class of finite linear orders. The relationship
between Fra\"{\i}ss\'{e} limits and cohesive powers is considered in (\cite%
{DHMSSV}). We now provide two more examples of structures isomorphic to
their cohesive powers.

\begin{example}
(1) $\Pi _{C}\left( \mathbb{Q}\times \mathbb{Z}\right) \cong \mathbb{Q}%
\times \mathbb{Z}$

(2) $\Pi _{C}\left( \mathbb{N+Q}\times \mathbb{Z}\right) \cong \mathbb{N+Q}%
\times \mathbb{Z}$
\end{example}

\begin{proof}
(1) $\Pi _{C}\mathbb{Q}\times \Pi _{C}\mathbb{Z}\cong \mathbb{Q}\times \Pi
_{C}(\mathbb{N}^{rev}\mathbb{+N})\cong \mathbb{Q}\times \left( \Pi _{C}%
\mathbb{N}^{rev}+\Pi _{C}\mathbb{N}\right) \cong $

$\cong \mathbb{Q}\times \lbrack \left( \mathbb{N+Q}\times \mathbb{Z}\right)
^{rev}+\left( \mathbb{N+Q}\times \mathbb{Z}\right) ]\cong \mathbb{Q}\times
\lbrack \mathbb{Q}\times \mathbb{Z+N}^{rev}+\mathbb{N+Q}\times \mathbb{Z}%
]\cong $

$\cong \mathbb{Q}\times \lbrack \mathbb{Q}\times \mathbb{Z+Z+Q}\times
\mathbb{Z}]\cong \mathbb{Q}\times \lbrack \mathbb{Q}\times \mathbb{Z]}\cong
\mathbb{Q}\times \mathbb{Z}\smallskip $

(2) $\Pi _{C}\left( \mathbb{N+Q}\times \mathbb{Z}\right) \cong \Pi _{C}%
\mathbb{N}+\Pi _{C}\left( \mathbb{Q}\times \mathbb{Z}\right) \cong \mathbb{%
N+Q}\times \mathbb{Z+Q}\times \mathbb{Z}\cong \mathbb{N+Q}\times \mathbb{Z}$
\end{proof}

Theorem \ref{Properties_LO}, part (4), demonstrates that having a computable
successor function is a sufficient condition for the cohesive power of a
computable linear order of type $\omega $ to be isomorphic to $\mathbb{N}+%
\mathbb{Q}\times \mathbb{Z}$. The next theorem shows that this condition is
not necessary.

\begin{theorem}
\label{noncomputable_succ}There is a computable linear order $\mathcal{L}$
of order type $\omega $ with a non-computable successor function such that
for every cohesive set $C$ we have $\Pi _{C}\mathcal{L}\cong \mathbb{N}+%
\mathbb{Q}\times \mathbb{Z}$.
\end{theorem}

\begin{proof}
Fix a non-computable c.e.\ set $A$, and let $f$ be a total computable
injection on the set of natural numbers with range $A$. Let $\mathcal{L}%
=(\omega ,<_{\mathcal{L}})$ be the linear order obtained by ordering the
even numbers according to their natural order, and by setting $2a<_{\mathcal{L}}2k+1<_{\mathcal{L}}2a+2$ if and only if $f(k)=a$. Specifically, we set
\begin{align*}
2c& <_{\mathcal{L}}2d & & \leftrightarrow & 2c& <2d \\
2c& <_{\mathcal{L}}2k+1 & & \leftrightarrow & c& \leq f(k) \\
2k+1& <_{\mathcal{L}}2c & & \leftrightarrow & f(k)& <c \\
2k+1& <_{\mathcal{L}}2\ell +1 & & \leftrightarrow & f(k)& <f(\ell ).
\end{align*}%
Then $\mathcal{L}$ is a computable linear order of type $\omega $. Let $S^{%
\mathcal{L}}$ denote the successor function of $\mathcal{L}$. Then $A\leq _{%
\mathrm{T}}S^{\mathcal{L}}$ (indeed, $A\equiv _{\mathrm{T}}S^{\mathcal{L}}$)
because $a\in A$ if and only if $S^{\mathcal{L}}(2a)\neq 2a+2$. Thus $S^{%
\mathcal{L}}$ is not computable.

Let $C$ be cohesive, and let $\mathcal{P}=\Pi _{C}\mathcal{L}$. We show that
$\mathcal{P}\cong \mathbb{N}+\mathbb{Q}\times \mathbb{Z}$. To do this, we
begin by establishing the following properties of $\mathcal{P}$.

\begin{enumerate}[(a)]

\item \label{it:PowOmega} $\mathcal{P}$ has an initial segment of type $%
\omega$.

\item \label{it:PowSucc} Every element of $\mathcal{P}$ has a $<_{\mathcal{P}}$-immediate successor.

\item \label{it:PowPred} Every element of $\mathcal{P}$ that is not the
least element has an $<_{\mathcal{P}}$-immediate predecessor.
\end{enumerate}

For~(\ref{it:PowOmega}), note that the range of the canonical embedding of $%
\mathcal{L}$ into $\mathcal{P}$ is an initial segment of $\mathcal{P}$ of
type $\omega $.

For~(\ref{it:PowSucc}), consider a $[\psi ]\in \mathcal{P}$. We define a
partial computable $\varphi $ such that, for almost every $n\in C$, $\varphi
(n)$ is the $<_{\mathcal{L}}$-immediate successor of $\psi (n)$. It then
follows that $[\varphi ]$ is the $<_{\mathcal{P}}$-immediate successor of $%
[\psi ]$. To define $\varphi $, observe that, by the cohesiveness of $C$,
exactly one of the following three cases occurs.

\begin{enumerate}
\item \label{it:AllOdd} $(\forall^\infty n \in C)(\text{$\psi(n)$ is odd})$

\item \label{it:AllInA} $(\forall^\infty n \in C)(\exists a \in A)(\psi(n) =
2a)$

\item \label{it:AllOutA} $(\forall^\infty n \in C)(\exists a \notin
A)(\psi(n) = 2a)$
\end{enumerate}

Note that we cannot effectively decide which case occurs, but in each case
we can define a particular $\varphi _{i}$ such that $[\varphi _{i}]$ is the $%
<_{\mathcal{P}}$-immediate successor of $[\psi ].$

If case~(\ref{it:AllOdd}) occurs, define
\begin{equation*}
\varphi _{1}(n)=%
\begin{cases}
2a+2 & \text{if $\psi (n){\downarrow }$, $\psi (n)=2k+1$, and $f(k)=a;$} \\
\uparrow & \text{otherwise.}%
\end{cases}%
\end{equation*}

If case~(\ref{it:AllInA}) occurs, define
\begin{equation*}
\varphi _{2}(n)=%
\begin{cases}
2k+1 & \text{if $\psi (n){\downarrow }$, $\psi (n)=2a$, $a\in A$, and $%
f(k)=a;$} \\
\uparrow & \text{otherwise.}%
\end{cases}%
\end{equation*}

If case~(\ref{it:AllOutA}) occurs, define
\begin{equation*}
\varphi _{3}(n)=%
\begin{cases}
2a+2 & \text{if $\psi (n){\downarrow }$ and $\psi (n)=2a;$} \\
\uparrow & \text{otherwise.}%
\end{cases}%
\end{equation*}%
In each case (i) ($i=1,2,3)$ we have that for almost every $n\in C$, $%
\varphi _{i}(n)$ is the $<_{\mathcal{L}}$-immediate successor of $\psi (n)$%
.\medskip

The proof of~(\ref{it:PowPred}) is analogous to the proof of~(\ref%
{it:PowSucc}).\medskip

For $[\psi ],[\varphi ]\in \mathcal{P}$, write $[\psi ]\ll _{\mathcal{P}%
}[\varphi ]$ if $[\psi ]<_{\mathcal{P}}[\varphi ]$ and the interval $([\psi
],[\varphi ])_{\mathcal{P}}$ in $\mathcal{P}$ is infinite. Using the
cohesiveness of $C$, we check that $[\psi ]\ll _{\mathcal{P}}[\varphi ]$ if
and only if $[\psi ]<_{\mathcal{P}}[\varphi ]$ and $\limsup_{n\in C}|(\psi
(n),\varphi (n))_{\mathcal{L}}|=\infty $, where $|(a,b)_{\mathcal{L}}|$
denotes the cardinality of the interval $(a,b)_{\mathcal{L}}$ in $\mathcal{L}
$. Notice that for even numbers $2a$ and $2b$, $2a<_{\mathcal{L}}2b$ if and
only if $2a<2b$. Therefore, if $2a<2b$, then $|(2a,2b)_{\mathcal{L}}|\geq
b-a-1$.

To finish the proof, we show the following.

\begin{enumerate}[(a)] \setcounter{enumi}{3}

\item \label{it:DenseBlocks} If $[\psi], [\varphi] \in \mathcal{P}$ satisfy $%
[\psi] \ll_{\mathcal{P}} [\varphi]$, then there is a $[\theta] \in \mathcal{P%
}$ such that $[\psi] \ll_{\mathcal{P}} [\theta] \ll_{\mathcal{P}} [\varphi]$.

\item \label{it:UnboundedBlocks} If $[\psi ]\in \mathcal{P}$, then there is
a $[\varphi ]\in \mathcal{P}$ with $[\psi ]\ll _{\mathcal{P}}[\varphi ]$%
.\medskip
\end{enumerate}

For~(\ref{it:DenseBlocks}), suppose that $[\psi ],[\varphi ]\in \mathcal{P}$
satisfy $[\psi ]\ll _{\mathcal{P}}[\varphi ]$. By (again) considering the
cases~(\ref{it:AllOdd})--(\ref{it:AllOutA}) above, either $\psi (n)$ is odd
for almost every $n\in C$, or $\psi (n)$ is even for almost every $n\in C$.
In the case where $\psi (n)$ is odd for almost every $n\in C$, $\widehat{%
\psi }(n)$ is even for almost every $n\in C$, where $[\widehat{\psi }]$ is
the $<_{\mathcal{P}}$-immediate successor of $[\psi ]$. Thus we may assume
that $\psi (n)$ and $\varphi (n)$ are even for almost every $n\in C$ by
replacing $[\psi ]$ and $[\varphi ]$ by their $<_{\mathcal{P}}$-immediate
successors if necessary. The condition $\limsup_{n\in C}|(\psi (n),\varphi
(n))_{\mathcal{L}}|=\infty $ is now equivalent to $\limsup_{n\in C}(\varphi
(n)-\psi (n))=\infty $.

Define a partial computable $\theta $ by
\begin{equation*}
\theta (n)=%
\begin{cases}
\left\lfloor \frac{\psi (n)+\varphi (n)}{2}\right\rfloor & \text{if $%
\left\lfloor \frac{\psi (n)+\varphi (n)}{2}\right\rfloor $ is even;} \\
&  \\
\left\lfloor \frac{\psi (n)+\varphi (n)}{2}\right\rfloor +1 & \text{if $%
\left\lfloor \frac{\psi (n)+\varphi (n)}{2}\right\rfloor $ is odd.}%
\end{cases}%
\end{equation*}

By the definition of $\theta $, we have that $\limsup_{n\in C}(\theta
(n)-\psi (n))=\infty $ and that $\limsup_{n\in C}(\varphi (n)-\theta
(n))=\infty $. Since $\psi (n)$, $\varphi (n)$, and $\theta (n)$ are even
for almost all $n\in C$, we have that:
\begin{equation*}
\limsup_{n\in C}|(\psi (n),\theta (n))_{\mathcal{L}}|=\infty \text{ and }%
\limsup_{n\in C}|(\theta (n),\varphi (n))_{\mathcal{L}}|=\infty .
\end{equation*}

~Thus, $[\psi ]\ll _{\mathcal{P}}[\theta ]\ll _{\mathcal{P}}[\varphi ],$ as
desired.\medskip

For~(\ref{it:UnboundedBlocks}), consider $[\psi ]\in \mathcal{P}$. As argued
above, we may assume that $\psi (n)$ is even for almost every $n\in C$ by
replacing $[\psi ]$ by its $<_{\mathcal{P}}$-immediate successor, if
necessary. If $\limsup_{n\in C}\psi (n)$ is finite, then by the cohesiveness
of $C$, the function $\psi $ must be eventually constant on $C$. In this
case, $[\psi ]\ll _{\mathcal{P}}[2\mathrm{id}]$. If $\limsup_{n\in C}\psi
(n)=\infty $, then $[\psi ]\ll _{\mathcal{P}}[2\psi ]$.\medskip

This completes the proof since the properties~(\ref{it:PowOmega})--(\ref%
{it:UnboundedBlocks}) ensure that $\mathcal{P}\cong \mathbb{N}+\mathbb{Q}%
\times \mathbb{Z}$.
\end{proof}

\section{Non-Isomorphic Cohesive Powers of Isomorphic Structures}

\begin{theorem}
For every co-maximal set $C\subseteq \omega $ there exist two isomorphic
computable structures $\mathcal{A}$ and $\mathcal{B}$ such the cohesive
powers $\prod_{C}\mathcal{A}$ and $\prod_{C}\mathcal{B}$ are not isomorphic.
\end{theorem}

\begin{proof}
Note that it suffices to prove the theorem for an arbitrary co-maximal set
consisting of even numbers only. Indeed, if $C$ is an arbitrary co-maximal
set, then $C_{1}=\{2s\mid s\in C\}$ is also a co-maximal set, and for any
computable structure $\mathcal{M}$, we have $\prod_{C}\mathcal{M}\cong
\prod_{C_{1}}\mathcal{M}$. Then, if $\mathcal{M}_{0}$ and $\mathcal{M}_{1}$
are isomorphic computable structures such that $\prod_{C_{1}}\mathcal{M}%
_{0}\ncong \prod_{C_{1}}\mathcal{M}_{1}$, then $\prod\nolimits_{C}\mathcal{M}%
_{0}\ncong \prod\nolimits_{C}\mathcal{M}_{1}.$

Let $S=\{2s\mid s\in \omega \}$. Let $A\subseteq S$ be such that $A_{1}=S-A$
is infinite and c.e.\ For every such $A$ we will define a computable
structure $\mathcal{M}_{A}$ with a single ternary relation.

Let $F=\{4s+1\mid s\in \omega \}$ and $B=\{4s+3\mid s\in \omega \}$. Fix a
computable bijection $f$ from the set $\{\left\langle i,j\right\rangle \in
S\mid i<j\}$ onto $F$. Let also $b$ be a computable bijection from the set $%
\{\left\langle j,i\right\rangle \in S\mid i<j\wedge (i\in A_{1}\vee j\in
A_{1})\}$ onto $B$. For the function $f$, we write $f_{ij}$ instead of $%
f(i,j)$ and similarly for the function $b$. Define a ternary relation $P$ as
follows:
\begin{eqnarray*}
P &=&\{\left( x,f_{xy},y\right) \mid x,y\in S\wedge x<y\}\cup \\
&&\{\left( y,b_{yx},x\right) \mid x,y\in S\wedge x<y\wedge (x\in A_{1}\vee
y\in A_{1})\}.
\end{eqnarray*}%
Finally, let $\mathcal{M}_{A}=\left\langle \omega ;P\right\rangle $.
Informally, we can view the triples $x,w,y$ with the property $P(x,w,y)$ as
labelled arrows (e.g., $x\overset{w}{\longrightarrow }y)$. We start with a
structure consisting of the set $S\cup F$ with arrows $i\overset{f_{ij}}{%
\longrightarrow }j$, that connect $i$ with $j$ for all $i,j\in S$ such that $%
i<j$. These arrows can be viewed as a way of redefining the natural ordering
$<$ on $S$. Elements of $S$ can be thought of as ``stem elements'' and
elements of $F$ can be thought of as ``forward witnesses.'' Next, we start
enumerating the c.e. set $A_{1}=S-A$. At every stage a new element $k$ is
enumerated into $A_{1}$, we add new arrows together with appropriate
elements from $B$, the ``backward witnesses,'' which intend to exclude $k$
from the initial ordering on $S$. More precisely, we add arrows $k\overset{%
b_{ki}}{\longrightarrow }i$ for all $i$ with $i<k,$ and arrows $j\overset{%
b_{jk}}{\longrightarrow }k$, for each $j$ with $j>k$. Eventually, exactly
the elements of $A_{1}$ will be excluded from the ordering, and the final
ordering will be an ordering on the set $A.$

In the resulting structure, every element $x\in A_{1}$ is connected with
every element $y\in S$ such that $x\neq y$ with exactly two arrows: $x%
\overset{w}{\longrightarrow }y$ and $y\overset{w_{1}}{\longrightarrow }x$.
If $x,y\in A$ are such that $x\neq y$ then they are connected with arrows of
the type $x\overset{w}{\longrightarrow }y$ exactly when $x<y$. In other
words, the formula
\begin{equation*}
\Phi (x,y)=_{def}\exists wP(x,w,y)\wedge \lnot \exists w_{1}P(y,w_{1},x)
\end{equation*}%
will be satisfied by exactly those $x,y\in A$ such that $x<y.$ The formula $%
\Phi $ will not be satisfied by any pair $(x,y)$ for which at least one of $%
x $ or $y$ has been excluded.

The following properties of the structure $\mathcal{M}_{A}$ follow
immediately from the definition above.

(1) For every $w$ there is at most one pair $x,y$ such that $P(x,w,y)$.

(2) If $x\in S-A,$ then for any $y\in S, \ y\not= x,$ there is a unique $%
w_{1}$ such that $P(x,w_{1},y)$ and a unique $w_{2}$ such that $P(y,w_{2},x)$%
.

(3) If $x,y\in A,$ then $x<y\Leftrightarrow \exists wP(x,w,y)$.

(4) $\mathcal{M}_{A}$ is computable.

To prove (4) note that the relation $P$ is computable because
\begin{equation*}
P(x,z,y)\Leftrightarrow x,y\in S\wedge \left[ (x<y\wedge z=f_{xy})\vee
(x>y\wedge z\in B\wedge b^{-1}(z)=\left\langle x,y\right\rangle )\right] .
\end{equation*}

(5) Let $D$, $E\subseteq S$ be infinite and such that $S-D$ and $S-E$ are
infinite and c.e. Then $\mathcal{M}_{D}\cong \mathcal{M}_{E}$.

Since $D$ and $E$ are infinite, the orders $\left( D,<\right) $ and $\left(
E,<\right) $, where $<$ is the natural order, are isomorphic to $\mathbb{N}$%
. The isomorphism between these orders, extended by any bijection between $%
S-D$ and $S-E,$ has a unique natural extension to a map from the domain of $%
\mathcal{M}_{D}$ to the domain of $\mathcal{M}_{E}$. That is, the arrows in $%
\mathcal{M}_{D}$ (the elements of $F$ and $B)$ can be uniquely mapped to
corresponding arrows in $\mathcal{M}_{E}$.\smallskip

To continue with the proof, we let
\begin{equation*}
\Theta (x)=_{def}\left( \exists t\right) \left[ \Phi (x,t)\vee \Phi (t,x)%
\right] .
\end{equation*}
The formula $\Theta (x)$ defines the set $A$ in $\mathcal{M}_{A}.$

For any structure $\mathcal{M}\mathfrak{=}\left( M,P\right) $ in the
language with one ternary predicate symbol we will use the following
notation:

$L_{\mathcal{M}}=_{def}\left\{ x\in M|\mathcal{M}\vDash \Theta (x)\right\} ,$%
and

$<_{L_{\mathcal{M}}}=_{def}\left\{ \left( x,y\right) \in M\times M|\mathcal{M%
}\vDash \Phi (x,y)\right\} .$

Fix $A\subseteq S$ such that $S-A$ is infinite and c.e.

It follows from the discussion above that the formula $\Phi (x,y)$ defines
in $\mathcal{M}_{A}$ the restriction of the natural order $<$ to $A$.
Clearly, $\left( L_{\mathcal{M}_{A}},<_{L_{\mathcal{M}_{A}}}\right) $ has
order type $\omega $.\medskip

Let $\mathcal{M}_{A}^{\sharp }=\prod_{C}\mathcal{M}_{A}.$ For partial
computable functions $g$ and $h$ such that $\left[ g\right] ,\left[ h\right]
\in \mathrm{dom}(\mathcal{M}_{A}^{\sharp })$ we have:

(i) $\mathcal{M}_{A}^{\sharp }\models \Phi (\left[ g\right] ,\left[ h\right]
)\Leftrightarrow C\subseteq ^{\ast }\{i|\left( g(i)\in A\right) \wedge
\left( h(i)\in A\right) \wedge \left( g(i)<h(i)\right) \}$

(ii) $L_{\mathcal{M}_{A}^{\sharp }}=\{\left[ g\right] \in \mathcal{M}%
_{A}^{\sharp }|$ $g(C)\subseteq ^{\ast }A\}$ and $\left( L_{\mathcal{M}%
_{A}^{\sharp }},<_{L_{\mathcal{M}_{A}^{\sharp }}}\right) $ is a linear order$%
.$

Note that (i) follows from Theorem \ref{FTCP}, part (2), since $\Phi (x,y)$
is a Boolean combination of $\Sigma _{1}^{0}$ and $\Pi _{1}^{0}$ formulas.

For the proof of (ii) notice that for any $\left[ g\right] \in \mathcal{M}%
_{A}^{\sharp }$ we have either $C\subseteq ^{\ast }\{i|g(i)\in A\}$ or $%
C\subseteq ^{\ast }\{i|g(i)\in \omega -A\}$ because $C$ is cohesive and $%
\omega -A$ is c.e. Since
\begin{equation*}
\left[ g\right] \in L_{\mathcal{M}_{A}^{\sharp }}\Leftrightarrow \left(
\exists x\right) \left[ \Phi (\left[ g\right] ,x)\vee \Phi (x,\left[ g\right]
)\right] ,
\end{equation*}%
the equivalence in part (i) implies that $L_{\mathcal{M}_{A}^{\sharp }}=\{%
\left[ g\right] \in \mathcal{M}_{A}^{\sharp }|$ $g(C)\subseteq ^{\ast }A\}.$
It is easy to show that the relation $<_{L_{\mathcal{M}_{A}^{\sharp }}}$ is
a linear order on $L_{\mathcal{M}_{A}^{\sharp }}$.

For any $a\in A$ let $h_{a}(i)=a$ for all $i\in \omega $. We will call the
element $\left[ h_{a}\right] $ in $\mathcal{M}_{A}^{\sharp }$ a constant in $%
\mathcal{M}_{A}^{\sharp }$.

(6) The set of constants $\{\left[ h_{a}\right] |a\in A\}$ in the structure $%
\mathcal{M}_{A}^{\sharp }$ forms an initial segment of $\left( L_{\mathcal{M}%
_{A}^{\sharp }},<_{L_{\mathcal{M}_{A}^{\sharp }}}\right) $ of order type $%
\omega $.

Clearly, if $a_{0},a_{1}\in A,$ then $\Phi (\left[ h_{a_{0}}\right] ,\left[
h_{a_{1}}\right] )$ if and only if $a_{0}<a_{1}$. Therefore, $\{\left[ h_{a}%
\right] |a\in A\}$\ is an ordered set of type $\omega $. It remains to check
that $\{\left[ h_{a}\right] |a\in A\}$ is an initial segment. Suppose $\left[
h\right] \in \mathcal{M}_{A}^{\sharp }$ and $a\in A$ are such that $\mathcal{%
M}_{A}^{\sharp }\vDash \Phi (\left[ h\right] ,\left[ h_{a}\right] ).$ Then
\begin{equation*}
C\subseteq ^{\ast }\{i|\mathcal{M}_{A}\vDash \Phi (h(i),a)\}=\{i|h(i)\in
A\wedge h(i)<a\}=\bigcup_{k\in A\wedge k<a}\{i|h(i)=k\}.
\end{equation*}%
The last expression is a union of a finite family of mutually disjoint c.e.
sets. Since $C$ is cohesive, there exists a $k\in A$ such that $C\subseteq
^{\ast }\{i|h(i)=k\}$, which means that $\left[ h\right] =\left[ h_{k}\right]
$ is a constant.

We now define the following $\Sigma _{3}^{0}$ sentence
\begin{equation*}
\Psi =_{def}\left( \exists x\right) \left[ \Theta (x)\wedge \left( \forall
y\right) \left[ \Theta (y)\Rightarrow \Phi (y,x)\right] \right] .
\end{equation*}%
The intended interpretation of $\Psi $ is that when $\Phi (x,t)$ defines a
linear order $\left( L_{\mathcal{M}},<_{L_{\mathcal{M}}}\right) ,$ then the
order has a greatest element$.$ Note that $\mathcal{M}_{A}\vDash \urcorner
\Psi .$ This is because $\left( L_{\mathcal{M}_{A}},<_{L_{\mathcal{M}%
_{A}}}\right) $ has order type $\omega $ and hence has no greatest element.

Before we continue with the proof we recall Proposition 2.1 from \cite{L}.

\begin{proposition}
\label{Alternative} (Lerman \cite{L}) Let $R$ be a co-$r$-maximal set, and
let $f$ be a computable function such that $f(R)\cap R$ is infinite. Then
the restriction $f\upharpoonright R$ differs from the identity function only
finitely.
\end{proposition}

We now fix a co-maximal (hence co-$r$-maximal) set $C\subseteq S$ and an
infinite co-infinite computable set $D\subseteq S$. By property (5) above,
we have $\mathcal{M}_{C}\cong \mathcal{M}_{D}$. Let $\mathcal{M}_{C}^{\sharp
}=\prod_{C}\mathcal{M}_{C}$ and $\mathcal{M}_{D}^{\sharp }=\prod_{C}\mathcal{%
M}_{D}.$

It is not hard to show that, since $C$ is co-maximal, for every partial
computable function $\varphi $ for which $C\subseteq ^{\ast }dom(\varphi )$,
there is a computable function $f_{\varphi }$ such that $[\varphi
]=[f_{\varphi }]$ (see \cite{DHMM}).

To finish the proof we will establish the following facts:

(7) $\mathcal{M}_{C}^{\sharp }\vDash \Psi $

(8) $\mathcal{M}_{D}^{\sharp }\vDash \urcorner \Psi $

To prove (7) recall that $L_{\mathcal{M}_{C}^{\sharp }}=\{\left[ f\right]
\in \mathcal{M}_{C}^{\sharp }|$ $f(C)\subseteq ^{\ast }C\}.$ By Proposition %
\ref{Alternative} if $\left[ f\right] \in \mathcal{M}_{C}^{\sharp }$ is such
that $f(C)\subseteq ^{\ast }C$ and $f(C)$ is infinite, then $\left[ f\right]
=\left[ \mathrm{id}\right] .$ If $f(C)$ is finite, then $f$ is eventually a
constant on $C,$ because $C$ is cohesive. Therefore, $L_{\mathcal{M}%
_{C}^{\sharp }}=\{\left[ f_{c}\right] \mid c\in C\}\cup \{\left[ \mathrm{id}%
\right] \}.$ It is easy to see that if $c\in C,$ then $\Phi (\left[ f_{c}%
\right] ,\left[ id\right] ).$ Thus, $\left( L_{\mathcal{M}_{C}^{\sharp
}},<_{L_{\mathcal{M}_{C}^{\sharp }}}\right) $ has order type $\omega +1$
with the greatest element $\left[ id\right] $. Therefore, $\mathcal{%
M}_{C}^{\sharp }\vDash \Psi .$

To prove (8), let $D=\{d_{0}<d_{1}<\cdots \}$. The function $g$ defined as $%
g(d_{i})=d_{i+1}\ $is computable. Suppose that $\mathcal{M}_{D}^{\sharp
}\vDash \Psi $ and let $\left[ f\right] $ be the greatest element in $\left(
L_{\mathcal{M}_{D}^{\sharp }},<_{L_{\mathcal{M}_{D}^{\sharp }}}\right) .$
Since $\left[ f\right] <_{L_{\mathcal{M}_{D}^{\sharp }}}\left[ g\circ f%
\right] ,$ it follows that $\mathcal{M}_{D}^{\ast }\vDash \urcorner \Psi .$%
\smallskip

In conclusion, we defined computable isomorphic structures $\mathcal{M}_{C}$
and $\mathcal{M}_{D}$ such that $\prod_{C}\mathcal{M}_{C}$ and $\prod_{C}%
\mathcal{M}_{D}$ are not even elementary equivalent. The structure $\mathcal{%
M}_{C}$ also provides a sharp bound for the fundamental theorem of cohesive
powers. Namely, for the $\Sigma _{3}^{0}$ sentence $\Psi ,$ $\mathcal{M}%
_{C}\vDash \urcorner \Psi $ but $\prod_{C}\mathcal{M}_{C}\vDash \Psi $.
\end{proof}

\section{Orders of type $\protect\omega $ with cohesive powers not
isomorphic to $\mathbb{N}+\mathbb{Q}\times \mathbb{Z}$}

We prove that if $C$ is co-maximal, then there is a computable linear order $%
\mathcal{L}$ of type $\omega $ (necessarily with a non-computable successor
function) such that $\Pi _{C}\mathcal{L}\ncong \mathbb{N}+\mathbb{Q}\times
\mathbb{Z}$.

\begin{lemma}
\label{lem-AvoidSucc} Let $C\subseteq \omega $ be co-c.e., infinite, and
co-infinite. Then there is a computable linear order $\mathcal{L}=(\omega
,<_{\mathcal{L}})$ of type $\omega $ such that for every partial computable
function $\varphi $,
\begin{equation}
\forall ^{\infty }n\in C(\varphi (n){\downarrow }\Rightarrow \text{$\varphi
(n)$ is not the $\mathcal{L}$-immediate successor of $n$}).  \tag{*}
\label{no_succ}
\end{equation}
\end{lemma}

\begin{proof}
Fix an infinite computable set $R\subseteq \overline{C}$. We define $<_{%
\mathcal{L}}$ in stages. By the end of stage $s$, $<_{\mathcal{L}}$ will
have been defined on $X_{s}\times X_{s}$ for some finite $X_{s}\supseteq
\{0,1,\dots ,s\}$. At stage $0$, set $X_{0}=\{0\}$ and define $0\nless _{%
\mathcal{L}}0$. At stage $s>0$, start with $X_{s}=X_{s-1}$ and update $X_{s}$
and $<_{\mathcal{L}}$ according to the following procedure.

\begin{enumerate}
\item If $<_{\mathcal{L}}$ has not yet been defined on $s$ (i.e., if $s
\notin X_{s}$), then update $X_{s}$ to $X_{s} \cup \{s\}$ and extend $<_{%
\mathcal{L}}$ to make $s$ the $<_{\mathcal{L}}$-greatest element of $X_{s}$.

\item \label{it:BreakSucc} Consider each $\langle e,n\rangle <s$ in order. If

\begin{enumerate}
\item $\varphi_{e,s}(n){\downarrow} \in X_{s}$,

\item $\varphi_{e}(n)$ is currently the $<_{\mathcal{L}}$-immediate
successor of $n$ in $X_{s}$,

\item \label{it:notR} $n \notin R$, and

\item \label{it:restraint} $n$ is not $<_{\mathcal{L}}$-below any of $0, 1,
\dots, e$,
\end{enumerate}

then let $m$ be the least element of $R-X_{s}$. Update $X_{s}$ to $X_{s}\cup
\{m\}$, and extend $<_{\mathcal{L}}$ so that $n<_{\mathcal{L}}m<_{\mathcal{L}%
}\varphi _{e}(n)$.
\end{enumerate}

This completes the construction.

We claim that for every $k$, there are only finitely many elements $<_{%
\mathcal{L}}$-below $k$. It follows that $\mathcal{L}$ is of type $\omega $.
Say that $\varphi _{e}$ \emph{acts for $n$ and adds $m$} when $<_{\mathcal{L}%
}$ is defined on an $m\in R$ to make $n<_{\mathcal{L}}m<_{\mathcal{L}%
}\varphi _{e}(n)$ as in~(\ref{it:BreakSucc}). Let $s_{0}$ be a stage with $%
k\in X_{s_{0}}$. Suppose at stage $s>s_{0}$ we add an $m$ to $X_{s}$ and
define $m<_{\mathcal{L}}k$. This can only be due to a $\varphi _{e}$ acting
for an $n\notin R$ and adding $m$ at stage $s$. At stage $s$, we must have $%
n<_{\mathcal{L}}k$ because $n<_{\mathcal{L}}m<_{\mathcal{L}}k$. Therefore,
we must also have $e<k$, for otherwise $k$ would be among $0,1,\dots e$, and
condition~(\ref{it:restraint}) would prevent the action of $\varphi _{e}$.
Furthermore, $m$ is chosen so that $m\in R$ and thus only elements of $R$
are added $<_{\mathcal{L}}$-below $k$ after stage $s_{0}$. Hence an $m$ can
only be added $<_{\mathcal{L}}$-below $k$ after stage $s_{0}$ when a $%
\varphi _{e}$ with $e<k$ acts for an $n<_{\mathcal{L}}k$ with $n\notin R$.
Each $\varphi _{e}$ acts at most once for every $n$, and no new $n\notin R$
appears $<_{\mathcal{L}}$-below $k$ after stage $s_{0}$. Thus, after stage $%
s_{0}$, only finitely many $m$ are ever added $<_{\mathcal{L}}$-below $k$.

We claim that for every $e$, (\ref{no_succ}) holds. Given $e$, let $\ell $
be the $<_{\mathcal{L}}$-greatest element of $\{0,1,\dots ,e\}$. Suppose
that $n>_{\mathcal{L}}\ell $ and $n\in C$. If $\varphi _{e}(n){\downarrow }$%
, let $s$ be large enough so that $\langle e,n\rangle <s$, $\varphi _{e,s}(n)%
{\downarrow }$, $n\in X_{s}$, and $\varphi _{e}(n)\in X_{s}$. Then either $%
\varphi _{e}(n)$ is already not the $\mathcal{L}$-immediate successor of $n$
at stage $s+1$, or at stage $s+1$ the conditions of~(\ref{it:BreakSucc}) are
satisfied for $\langle e,n\rangle $, and an $m$ is added such that $n<_{%
\mathcal{L}}m<_{\mathcal{L}}\varphi _{e}(n)$.
\end{proof}

\begin{theorem}
Let $C$ be a co-maximal set. Then there is a computable linear order $%
\mathcal{L}$ of type $\omega $ such that $[\mathrm{id}]$ does not have a
successor in $\Pi _{C}\mathcal{L}$. Therefore, $\Pi _{C}\mathcal{L}\ncong
\mathbb{N}+\mathbb{Q}\times \mathbb{Z}$.
\end{theorem}

\begin{proof}
Let $\mathcal{L}$ be a computable linear order as in Lemma~\ref%
{lem-AvoidSucc} for $C$. Suppose that $\varphi $ is a partial computable
function such that $[\mathrm{id}]<_{\Pi _{C}\mathcal{L}}[\varphi ]$. We show
that $[\varphi ]$ is not the $<_{\Pi _{C}\mathcal{L}}$-immediate successor
of $[\mathrm{id}]$. The inequality $[\mathrm{id}]<_{\Pi _{C}\mathcal{L}%
}[\varphi ]$ means that $\left( \forall ^{\infty }n\in C\right) (n<_{%
\mathcal{L}}\varphi (n))$. However, by Lemma~\ref{lem-AvoidSucc},
\begin{equation*}
\left( \forall ^{\infty }n\in C\right) (\text{$\varphi (n)$ is not the $%
\mathcal{L}$-immediate successor of $n$}).
\end{equation*}%
Define a partial computable $\psi $ so that, for every $n$,
\begin{equation*}
\psi (n)=%
\begin{cases}
\text{the least $m$ such that $n<_{\mathcal{L}}m<_{\mathcal{L}}\varphi (n)$}
& \text{if there is such an $m;$} \\
\uparrow & \text{otherwise}.%
\end{cases}%
\end{equation*}%
Then $\left( \forall ^{\infty }n\in C\right) (n<_{\mathcal{L}}\psi (n)<_{%
\mathcal{L}}\varphi (n))$. Thus, $[\mathrm{id}]<_{\Pi _{C}\mathcal{L}}[\psi
]<_{\Pi _{C}\mathcal{L}}[\varphi ]$. So, $[\varphi ]$ is not the $<_{\Pi _{C}%
\mathcal{L}}$-immediate successor of $[\mathrm{id}]$.

It follows that $\Pi _{C}\mathcal{L}\ncong \mathbb{N}+\mathbb{Q}\times
\mathbb{Z}$ because every element of $\mathbb{N}+\mathbb{Q}\times \mathbb{Z}$
has an immediate successor, but $[\mathrm{id}]\in \Pi _{C}\mathcal{L}$ does
not have an immediate successor.
\end{proof}

Note that the sentence $\Psi $ that states that every element has an
immediate successor is $\Pi _{3}^{0}$. Then for the computable linear order $%
\mathcal{L}$ of type $\omega $ constructed above, $\mathcal{L}\vDash \Psi $
but $\Pi _{C}\mathcal{L}\vDash \urcorner \Psi .$

\end{document}